\documentclass[11pt,dvips,twoside,letterpaper]{article}
\usepackage{pslatex}
\usepackage{fancyhdr}
\usepackage{graphicx}
\usepackage{geometry}
\RequirePackage[latin1]{inputenc} \RequirePackage[T1]{fontenc}

\def\figurename{Figure} 
\makeatletter
\renewcommand{\fnum@figure}[1]{\figurename~\thefigure.}
\makeatother

\def\tablename{Table} 
\makeatletter
\renewcommand{\fnum@table}[1]{\tablename~\thetable.}
\makeatother
\usepackage{color}
                [2000/03/29 v1.0m Standard LaTeX package]

\usepackage{amsmath}
\usepackage{amssymb}
\usepackage{amsfonts}
\usepackage{amsthm,amscd}
\newtheorem{theorem}{Theorem}[section]
\newtheorem{lemma}[theorem]{Lemma}

\theoremstyle{definition}

\newtheorem{definition}[theorem]{Definition}

\theoremstyle{remark}
\newtheorem{remark}[theorem]{Remark}

\numberwithin{equation}{section}

\def\P{\mathbb P}
\def\R{\mathbb R}
\def\E{\mathbb E}

\def\Q{\mathbb Q}

\def\E{\mathbb E}
\def\N{\mathbb N}
\def\cal{\mathcal}

\begin{document}

\title{Generalized backward doubly
stochastic differential equations driven by Lévy processes with continuous coefficients\thanks{The first author is supported by TWAS Research Grants to individuals (No. 09-100 RG/MATHS/AF/\mbox{AC-I}--UNESCO FR: 3240230311)}}

\author{A. Aman\thanks{augusteaman5@yahoo.fr, corresponding author}\, and J. M. Owo\thanks{owo jm@yahoo.fr} \vspace{0.2cm}\\
UFR de Math\'{e}matiques et Informatique\\ $22$ BP $582$ Abidjan, C\^{o}te d'Ivoire
\\ Universit\'{e} de Cocody}

\date{}
\maketitle \thispagestyle{empty} \setcounter{page}{1}

\thispagestyle{fancy} \fancyhead{}
 \fancyfoot{}
\renewcommand{\headrulewidth}{0pt}

\begin{abstract}
A new class of generalized backward doubly stochastic differential equations (GBDSDEs
in short) driven by Teugels martingales associated
with Lévy process are investigated. We establish a comparison theorem which allows us to derive an existence result of solutions
under continuous and linear growth conditions.
\end{abstract}%

\vspace{.08in} \noindent \textbf{MSC}:Primary: 60F05, 60H15; Secondary: 60J30\\
\vspace{.08in} \noindent \textbf{Keywords}: Backward doubly stochastic
differential equations; Lévy processes; Teugels martingales;
comparison theorem; continuous and linear growth conditions.

\section{Introduction}
Backward Stochastic Differential Equations (BSDEs) have been introduced (in the non-linear case) by Pardoux
and Peng \cite{PP1}. Originally, the study of the BSDEs has been motivated by its connection with partial differential equations (PDEs, in short). Indeed, BSDEs provides the probabilistic interpretation for solutions of both parabolic and elliptic semi linear partial differential equations generalizing the well-know Feynman-Kac formula (see Pardoux and Peng \cite{PP2}, Peng\cite{Pg}). Very quickly this kind of equations has gained importance because of their many applications in the theory of mathematical finance (El Karoui et al., \cite{ELK1}), in
stochastic control (El Karoui and Hamadène, \cite{KH}) and stochastic games (Hamadène and Lepeltier, \cite{HL2}). Roughly speaking, BSDEs is equation in the form:
\begin{eqnarray}
Y_t=\xi+\int_{t}^{T}f(s,Y_s,Z_s)ds-\int_{t}^{T}Z_sdW_s,
\label{initial}
\end{eqnarray}
where $f$ is the generator, $\xi$ is the terminal value and $W$ is the brownian motion. All of them are the given data. Denote by $(\cal{F}_t )_{0\leq t\leq T}$ the natural filtration generated by $W$, the solution of BSDE $(\xi,f)$ is the $\cal{F}_t$-adapted process $(Y, Z)$ satisfies \eqref{initial} and belongs in the appropriate space. In \cite{PP1}, Pardoux and Peng derived existence and uniqueness result to BSDE $(\xi, f)$ under uniformly Lipschitz generator. They used the martingale representation theorem which is the main tool in the theory of BSDEs. A few years later, further researches weak the Lipschitz condition.  Lepeltier and San Martin \cite{LSm} study BSDEs with continuous
coefficients, Kobylanski \cite{Koby} introduced BSDEs with the quadratic coefficients in $z$, Briand and Carmona \cite{BC} considered BSDEs with polynomial growth generators.

On the other hand, applying the idea used in \cite{PP1}, Pardoux and Peng introduced in \cite{PardPeng}
the so-called backward doubly stochastic differential equations (BDSDEs, in short). This kind of BDSDEs gives a probabilistic representation for a class of quasilinear stochastic partial differential equations (SPDEs, in short). Next, Bally and Matoussi \cite{BMat} used also BDSDEs to give the probabilistic representation of the weak solutions of parabolic semi linear SPDEs in Sobolev spaces; Matoussi and Scheutzow \cite{AM} introduced another kind of BDSDEs to derive a probabilistic representation for the solution of SPDEs with nonlinear noise term given by the Itô-Kunita stochastic integral.; Boufoussi et al. \cite{Boufsi} recommended a class of generalized BDSDEs (GBDSDEs, in short) which involved an integral with respect to an adapted continuous increasing process and gave the probabilistic representation for stochastic viscosity solutions of semi-linear SPDEs with a Neumann boundary condition.

In \cite{NS}, Nualart and Schoutens proved a martingale representation theorem
associated to a Lévy process. This progress allows them to establish  in \cite{NSc} the existence and
uniqueness result for BSDEs associated with a Lévy process. In continuation of all this works, Ren et al. \cite{Ren} showed existence
and uniqueness result to GBDSDEs driven by Lévy process (GBDSDEL, in short) under Lipschitz on the generator. Moreover, the probabilistic interpretation for solutions of a class of stochastic partial differential integral equations (SPDIEs, in short) with a nonlinear Neumann boundary condition has been established.

In this note, we consider GBDSDEL
\begin{eqnarray}
Y_{t}&=&\xi+\int_{t}^{T}f(s,Y_{s^-},Z_{s})ds+\int_{t
}^{T}h(s,Y_{s^-})dA_s+\int_{t}^{T}g(s,Y_{s^-})\,d\overleftarrow{B}_{s}\nonumber\\
&&-\sum_{i=1}^{\infty}\int_{t}^{T}Z^{(i)}_{s}dH^{(i)}_{s},\,\ 0\leq t\leq
T.\label{a111}
\end{eqnarray}
More precisely, we establish the existence result to BDSDEs \eqref{a011} under continuous condition on the generators. The proof is strongly linked to the comparison theorem which do not hold in general case (see \cite{Aal} for BDSDE and  the counter-example in \cite{Barles} for BSDEs driven by Lévy processes).
To overcome this difficulty, we assume relation \eqref{adprop} between the generator $f$ and Lévy process $L$ which have only $m$ different jump size with no continuous part.

The rest of the paper is organized as follows. In section 2, we introduce some
preliminaries and deal with a comparison theorem for GBSDEL under Lipschitz generators.
Section 3 is devoted to prove the existence
result to GBDSDEs driven by Lévy processes under continuous generators.
\section{Preliminaries}
\subsection{Notations and Definition}
Let $(\Omega, \mathcal{F},\mathbb{P})$ be a complete probability
space on which are defined all the processes stated in this paper and $T$ be
a
fixed final time.
\\
Let  $\{B_{t}; 0\leq t\leq T \}$ be a standard
Brownian motion, with values in $\mathbb{R}^{}$ and $%
\{L_{t}; 0\leq t\leq T \}$ be a $\mathbb{R}%
^{}$-valued Lévy process independent of $\{B_{t}; 0\leq t\leq T \}$ corresponding to a standard Lévy measure $\nu$ such that $\int_{\R}(1\wedge y)\nu(dy)<\infty$.
\\
Let $\mathcal{N}$ denote the class of $\P$-null sets
of $\mathcal{F}$. For each $t \in [0,T]$, we define
\begin{equation*}
\mathcal{F}_{t}\overset{\Delta}{=}\mathcal{F}_{t}^{L} \vee \mathcal{F}%
_{t,T}^{B},
\end{equation*}
where for any process $\{\eta_{t}\}$ ; $\mathcal{F}_{s,t}^{\eta}=\sigma
\{\eta_{r}-\eta_{s}; s\leq r \leq t \} \vee \mathcal{N}$, $\mathcal{F}%
_{t}^{\eta}=\mathcal{F}_{0,t}^{\eta}$. \newline
Note that $\{\mathcal{F}_{t}^L,\, t\in [0,T]\}$ is an increasing filtration
and $\{\mathcal{F}_{t,T}^B,\, t\in [0,T]\}$ is a decreasing filtration. Thus
the collection $\{\mathcal{F}_{t},\, t\in [0,T]\}$ is neither increasing nor
decreasing so it does not constitute a filtration.

In the sequel, $\{A_{t}; 0\leq t\leq T \}$
is a $\mathcal{F}_{t}$-measurable, continuous and increasing real valued process such that $A_0=0$.

Let us introduce some spaces:\newline
For any $m\geq1$, $\mathcal{M}^{2}(0,T,\mathbb{R}^{m})$ denotes the space of $\R^{m}$-valued random process satisfying:
\begin{enumerate}
\item[(i)] $\|\varphi \|_{\mathcal{M}^{2}(\R^{m})}^{2}=\sum_{i=1}^{m}\mathbb{E}%
\left(\int_{0}^{T}\mid \varphi_{t}^{(i)} \mid^{2} dt\right)< \infty$

\item[(ii)] $\varphi_{t}$ is $\mathcal{F}_{t}$-jointly measurable, for any $t \in
[0,T].$
\end{enumerate}
Similarly, $\mathcal{S}^{2}(0,T)$
stands for the set of real valued random processes
which satisfy:
\begin{enumerate}
\item[(i)] $\|\varphi \|_{\mathcal{S}^{2}}^{2}=\mathbb{E}\left(%
\underset{0\leq t\leq T}{\sup} \mid \varphi_{t}\mid^{2}\right)< \infty$

\item[(ii)] $\varphi_{t}$ is $\mathcal{F}_{t}$-measurable, for any $t \in
[0,T].$
\end{enumerate}
$\mathcal{A}^{2}(0,T)$ denotes
the set of (class of $d\P\otimes dA_t$ a.e. equal) real
valued measurable random processes $\{\varphi_{t}; 0\leq t\leq T \}$
such that
\begin{enumerate}
\item[(i)] $\|\varphi\|_{\mathcal{A}^2}^{2}=\displaystyle\mathbb{E}
\left(\int_{0}^{T}\mid \varphi_{t} \mid^{2} dA_t\right)< \infty.$
\item[(ii)] $\varphi_{t}$ is $\mathcal{F}_{t}$-measurable, for any $t \in
[0,T].$
\end{enumerate}
The space $\mathcal{E}_{m}(0,T)=\big(\mathcal{S}^{2}(0,T)
 \cap\mathcal{A}^{2}(0,T)\big)\times\mathcal{M}^{2}(0,T,\R^{m})$ endowed with norm
\begin{eqnarray*}
\|(Y,Z)\|_{\mathcal{E}_m}^{2}=\mathbb{E}\left(%
\underset{0\leq t\leq T}{\sup} \mid Y_{t}\mid^{2}
+\int_{0}^{T}\mid Y_{s} \mid^{2} dA_s+
\int_{0}^{T}\|Z_{s}\|^{2} ds\right)
\end{eqnarray*}
is a Banach space.

Furthermore, let consider the Teugels Martingale $(H^{(i)})_{i\geq1}$
associated with the Lévy process  $\{L_{t}; 0\leq t\leq T \}$ defined by:
\begin{eqnarray*}
H_{t}^{(i)}=c_{i,i}T_{t}^{(i)}+c_{i,i-1}T_{t}^{(i-1)}+...+c_{i,1}T_{t}^{(1)},
\end{eqnarray*}
where $T_{t}^{(i)}=L_{t}^{(i)}-\E(L_{t}^{(i)})=L_{t}^{(i)}-t\E(L_{1}^{(i)})$
for all $i\geq1$. Let remark that the process $L_{t}^{(i)}$ have power jump, for all $i\geq 1$. More precisely,
denoting $ \Delta L_{s}=  L_{s^{}}-L_{s^{-}}$, we have $L_{t}^{(1)}=L_{t}$ and $\displaystyle L_{t}^{(i)}=\sum_{0<s\leq t}( \Delta L_{s})^i$
for $i\geq2$. In \cite{NSc}, Nualart and Schoutens proved  that
the coefficients $c_{i,k}$ correspond to the orthonormalization
of the polynomials $1,\ x, \ x^2,\cdot\cdot\cdot$ with respect to the measure
$\mu(dx)=x^2\nu(dx)+\sigma^2\delta_0(dx)$, i.e $q_{i}(x)=c_{i,i}x^{i-1}+c_{i,i-1}x^{i-2}+...+c_{i,1}$.
The martingale $(H^{(i)})_{i\geq1}$ can be chosen to be
pairwise strongly orthonormal martingale. \\That is for all $i,j$,
\ $\displaystyle\langle H^{(i)},H^{(j)}\rangle_{t}=\delta_{ij}t$.
\begin{remark}
Since the Lévy process $L$ has only $m$ different jump size with no continuous part, the Teugels martingales $H^{(i)}=0,\; \forall\; i\geq m+1$.
In this context, BDSDEs \eqref{a111} can be write rigorously
\begin{eqnarray}
Y_{t}&=&\xi+\int_{t}^{T}f(s,Y_{s^-},Z_{s})ds+\int_{t
}^{T}h(s,Y_{s^-})dA_s+\int_{t}^{T}g(s,Y_{s^-})\,d\overleftarrow{B}_{s}\nonumber\\
&&-\sum_{i=1}^{m}\int_{t}^{T}Z^{(i)}_{s}dH^{(i)}_{s},\,\ 0\leq t\leq
T.\label{a011}
\end{eqnarray}
\end{remark}
\begin{definition}
A pair of $\mathbb{R} \times
\R^{m}$-valued process $(Y,Z)$ is called solution of GBDSDEL $(\xi, f,g, h, A)$
driven by Lévy processes if $(Y,Z)\in \mathcal{E}_{m}(0,T)$ and verifies $\eqref{a011}$.
\end{definition}
\subsection{GBDSDEL with Lipschitz coefficients}
For memory, let recall the existence and uniqueness result for GBDSDEL
under Lipschitz condition due to Ren et al., \cite{Ren}. Here, the function $g$ depends on $z$ and we have the following assumptions:
\begin{description}
\item \textbf{(A1)}\ \ The terminal value $\xi \in \mathrm{L}^{2}(\Omega,
\mathcal{F}_{T}, \mathbb{P}, \mathbb{R})$ such that for all $\lambda>0$,\ \
$\E (e^{\lambda A_T}|\xi|^2)<\infty,$
\noindent \item \textbf{(A2)}\ \ The generators $f, g:\Omega \times [0,T]\times
\mathbb{R} \times \R^{m}\rightarrow
\mathbb{R}$ and $h:\Omega \times [0,T]\times \mathbb{R}
\rightarrow \mathbb{R}$ satisfy, for $\beta_1,\ \beta_2\in\R$,
$K>0$, $0 < \alpha < 1$ and three $\mathcal{F}_t$-measurable processes
$\{f_t, g_t, h_t:0\leq t\leq T\}$ with values in $[1,\infty[$
and for all $(t,y,z)\in\Omega \times [0,T]\times
\mathbb{R} \times \R^{m}$, $\lambda >0$
\begin{itemize}
  \item [(i)] $f(.y,z), g(.y,z)$ and $h(.,y)$ are jointly measurable,
  \item [(ii)] $\left\{
             \begin{array}{ll}
               |f(t,y,z)|\leq f_t+ K(|y|+\|z\|), \vspace{0.1cm}& \hbox{} \\
               |g(t,y,z)|\leq g_t+ K(|y|+\|z\|), \vspace{0.1cm}& \hbox{} \\
               |h(t,y)|\leq h_t+K|y|, & \hbox{}
             \end{array}
           \right.$
  \item [(iii)] $\displaystyle\E (\int_{0}^{T}e^{\lambda A_t}f_t^2dt+
  \int_{0}^{T}e^{\lambda A_t}g_t^2dt+\int_{0}^{T}e^{\lambda A_t}h_t^2dA_t)<\infty,$

 \item [(iv)]  $\left\{
             \begin{array}{ll}
               |\langle y-y',f(t,y,z)-f(t,y',z)\rangle  \leq \beta_1\mid
y-y'\mid^{2},   \vspace{0.1cm}& \hbox{} \\
               \langle y-y',h(t,y)-h(t,y')\rangle  \leq \beta_2\mid
y-y'\mid^{2},  \vspace{0.1cm}& \hbox{}

             \end{array}
           \right.$
 \item [(v)] $\beta_2<0$,
\end{itemize}

\noindent \item \textbf{(A3)}\ \ $\left\{
             \begin{array}{ll}
             |f(t,y,z)-f(t,y',z')|^{2}
             \leq K( |y-y'|^{2}+\| z-z'\|^{2}),\vspace{0.1cm}& \hbox{}\\
               |g(t,y,z)-g(t,y',z')|^{2} \leq K|
y-y'|^{2}+\alpha \| z-z'\|^{2}, \vspace{0.1cm}& \hbox{}\\
                |h(t,y)-h(t,y')|^{2}\leq K|y-y'|^{2}.& \hbox{}
             \end{array}
           \right.$
\end{description}
\begin{theorem}[Ren et al. \cite{Ren}]\label{lm0}
Under the assumptions $({\bf A1})$-$({\bf A3})$, the GBDSDEL $\eqref{a011}$ has a unique solution.
\end{theorem}
\begin{remark}
\begin{description}
\item $(i)\;$ Whenever $(Y_t, Z_t)$ satisfies $(\ref{a011})$, $(\bar{Y}_t,\ \bar{Z}_t)=(e^{\lambda A_t}Y_t,\ e^{\lambda A_t}Z_t)$ satisfies
an analogous GBDSDEL, with $f$, $g$ and $h$ replaced by
\begin{eqnarray*}
\bar{f}(t,\ y,\ z)&=& e^{\lambda A_t}f(t,\ e^{-\lambda A_t}y,\ e^{-\lambda A_t}z)\\
\bar{g}(t,\ y,\ z)&=& e^{\lambda A_t}g(t,\ e^{-\lambda A_t}y,\ e^{-\lambda A_t}z)\\
\bar{h}(t,\ y)&=& e^{\lambda A_t}h(t,\ e^{-\lambda A_t}y)-\lambda y
\end{eqnarray*}
Hence, if $h$ satisfies $(iv)$ with a possibly non negative $\beta_2$, we can always
choose $\lambda$ such that $\bar{h}$ satisfies $(iv)$ with a strictly negative $\beta_2$. Consequently, $(v)$ is not a severe restriction.
\item $(ii)\;$ To assure the existence and uniqueness of the solution to the GBDSDEL \eqref{a011}, there is no need to have the assumptions $(iv)$ and $(v)$ of $({\bf A2})$. It is just needed to simplify the calculation in the proof of a priori estimate.
\end{description}
\end{remark}
\subsection{Comparison theorem}
The comparison theorem is one of the principal tools in the theory of the BSDEs which does not hold in general for BSDEs with jumps (see the counter-example in
Barles et al. \cite{Barles}). With a additional property of the jumps size \eqref{adprop} as in \cite{Q}, we derive the comparison theorem for GBDSDEs driven by Lévy processes under Lipschitz condition which generalizes the work of Yufeng et al. \cite{Y} for GBDSDEs with non jumps. In this fact, given $\xi^k$ and $f^k,\ h^k,\ g$ for $k=1,2$ we consider
\begin{eqnarray}\label{eq1l}
Y_t^{k}&=&\xi^{k}+\int_{t}^{T}f^{k}(s,Y_{s^{-}}^{k},Z_s^{k})ds+
\int_{t}^{T}h^k(s,Y_{s^{-}}^{k})dA_{s}+
\int_{t}^{T}g(s,Y_{s^{-}}^{k})\overleftarrow{dB_{s}}\notag\\
&&-\sum_{i=1}^{m}\int_{t}^{T}Z_s^{k(i)}dH_{s}^{(i)},\ \ t\in[0,T].
\end{eqnarray}
Under assumptions $({\bf A1}), \; ({\bf A2})$  and $({\bf A3})$, it follows from Theorem \ref{lm0} that $(Y^{k},Z^{k})$ is a unique solution of BDSDEL \eqref{eq1l}.
\begin{theorem}\label{tc}
Assume  $({\bf A1})$-$({\bf A3})$ and let $(Y^1,Z^1)$ and $(Y^2,Z^2)$ be the solutions of equations \eqref{eq1l} for $k=1,2$. We suppose:
\begin{itemize}
\item $\xi^1\geq \xi^2, \ \P$-a.s.
\item $f^1(t,y,z)\geq f^2(t,y,z)$, and $h^1(t,y)\geq h^2(t,y)$ \ $\P$-a.s.,
  for all $(t,y,z)\in [0,T]\times \R\times\R^m$,
\item $\displaystyle\beta^i_t=\frac{f^1(t,y^{2},\widetilde{z}^{(i-1)})-
 f^1(t,y^{2},\widetilde{z}^{(i)})}{z^{1(i)}- z^{2(i)}}
  \mathbf{1}_{\left\{z^{1(i)}\neq z^{2(i)}\right\}}$,
\end{itemize}
where
$\widetilde{z}^{(i)}=\Big(z^{2(1)},z^{2(2)},...,z^{2(i)},z^{1(i+1)},...,z^{1(m)}\Big)$
such that
\begin{eqnarray}
\sum_{i=1}^{m}\beta_t^i\Delta H_t^{(i)}>-1, \ dt\otimes d\P\mbox{-a.s}.\label{adprop}
\end{eqnarray}
Then, we have for all $t\in[0,T],\:\: Y_t^1\geq Y_t^2$, a.s.

Moreover, for all $(t,y,z)\in[0,T]\times\mathbb{R} \times \mathbb{R}^m$, if \ $ \xi^{1}> \xi^{2}$, or $f^{1}(t,y,z)> f^{2}(t,y,z)$, or $h^{1}(t,y)> h^{2}(t,y)$, a.s.,
$Y_{t}^{1}> Y_{t}^{2},\ \ a.s.,\ \forall\, t \in [0,T]$.
\end{theorem}
\begin{proof}
Set
\begin{eqnarray*}
a_t&=&
\frac{f^{1}(t,Y_{t^{-}}^{1},Z_{t}^{1})-f^{1}(t,Y_{t^{-}}^{2},Z_{t}^{1})}{
(Y_{t^{-}}^{1}-Y_{t^{-}}^{2})\textbf{1}_{\{Y_{t^{-}}^{1}\neq
Y_{t^{-}}^{2}\}}},\\
b_t&=&
\frac{h^1(t,Y_{t^{-}}^{1})-h^1(t,Y_{t^{-}}^{2})}{(Y_{t^{-}}^{1}-Y_{t^{-}}^{2})
\textbf{1}_{\{Y_{t^{-}}^{1}\neq
Y_{t^{-}}^{2}\}}},\\
c_t&=&\frac{g(s,Y_{t^{-}}^{1})-g(t,Y_{t^{-}}^{2})}{(Y_{t^{-}}^{1}-Y_{t^{-}}^{2})
\textbf{1}_{\{Y_{t^{-}}^{1}\neq Y_{t^{-}}^{2}\}}}.
\end{eqnarray*}
Next, it follows from $({\bf A3})$ that the processes $(a_t)_{t\in[0,T]}$,
$(b_t)_{t\in[0,T]}$ and  $(c_t)_{t\in[0,T]}$  are measurable and bounded.

Therefore, for $0\leq s \leq t \leq T$, the linear BDSDE
\begin{eqnarray*}
\Gamma_{s,t}=1+\displaystyle\int_s^t\Gamma_{s,r^{-}}dX_r,
\end{eqnarray*}
with
\begin{eqnarray*}
X_t=\int_{0}^{t}a_r dr+\int_{0}^{t}b_{r}dA_{r}+\int_{0}^{t}c_{r}\overleftarrow{dB_{r}}+\sum_{i=1}^{m}\int_{0}^{t}\beta_{r}^{i}dH_{r}^{(i)}
\end{eqnarray*}
have (cf. Doléans-Dade exponential formula) a unique $\mathcal{F}_t$-measurable solution
\begin{eqnarray}
 \Gamma_{s,t}&=&\exp\Big(\int_{s}^{t}a_r dr+\int_{s}^{t}b_{r}dA_{r}+\int_{s}^{t}c_{r}\overleftarrow{dB_{r}}-\frac{1}{2}\int_{s}^{t}|
c_r|^2dr\Big)\nonumber\\
&&\times\prod_{s<r\leq t}(1+\sum_{i=1}^{m}\beta_{r}^{i}\Delta H_{r}^{(i)})\exp\left(-\sum_{i=1}^{m}\beta_{r}^{i}\Delta H_{r}^{(i)}\right).
\end{eqnarray}

Further, denoting $\bar{\xi}=\xi^{1}-\xi^{2}, \, \, \bar{Y}_t
=Y_{t}^{1}-Y_{t}^{2}, \,\, \bar{Z}_t=Z_{t}^{1}-Z_{t}^{2},\,\,\bar{f}_{t}=f^{1}(t,Y_{t^{-}}^{2},Z_{t}^{2})-f^{2}(t,Y_{t^{-}}^{2},Z_{t}^{2})$ and $\bar{h}_{t}=h^{1}(t,Y_{t^{-}}^{2})-h^{2}(t,Y_{t^{-}}^{2})$, we have
\begin{eqnarray}\label{eq3l}
\bar{Y}_{t}&=&\bar{\xi}+\int_{t}^{T}[a_s \bar{Y}_{s^{-}}+ \sum_{i=1}^{m}\beta_s^i
\bar{Z}_{s}^{(i)}+\bar{f}_{s}]ds +\int_{t}^{T}[b_s
\bar{Y}_{s^{-}}+\bar{h}_s]dA_{s}+ \int_{t}^{T}c_s
\bar{Y}_{s^{-}}\overleftarrow{dB_{s}}\notag\\
&&-\sum_{i=1}^{m}\int_{t}^{T}\bar{Z}_{s}^{(i)}dH_{s}^{(i)},\ \ \ \ \ t\in[0,T].
\end{eqnarray}
Itô's formula to $\Gamma_{s,r}Y_r$ from $r = t$ to $r = T$ provides
\begin{eqnarray*}
\Gamma_{s,t}\bar{Y}_t&=&\Gamma_{s,T}\bar{\xi}-\int_{t}^{T}\Gamma_{s,r^{-}}d\bar{Y}_r
-\int_{t}^{T}\bar{Y}_{r^{-}}d\Gamma_{s,r^{}}
-
\int_{t}^{T}d[\Gamma,\bar{Y}]_{r}\\
&=&\Gamma_{s,T}\bar{\xi}+\int_{t}^{T}\Gamma_{s,r^{-}}
[a_r \bar{Y}_{r^{-}}+ \sum_{i=1}^{m}\beta_r^i
\bar{Z}_{r}^{(i)}+\bar{f}_{r}]dr+\int_{t}^{T}\Gamma_{s,r^{-}}[b_r
\bar{Y}_{r^{-}}+\bar{h}_r]dA_{r}\\
&&+\int_{t}^{T}\Gamma_{s,r^{-}}c_r
\bar{Y}_{r^{-}}\overleftarrow{dB_{r}}-
\sum_{i=1}^{m}\int_{t}^{T}\Gamma_{s,r^{-}}\bar{Z}_{r}^{(i)}dH_{r}^{(i)}
\notag\\&& -\int_{t}^{T}\bar{Y}_{r^{-}}\Gamma_{s,r^{-}}a_{r}dr-\int_{t}^{T}\bar{Y}_{r^{-}}\Gamma_{s,r^{-}}b_{r}dA_r
-\int_{t}^{T}\bar{Y}_{r^{-}}\Gamma_{s,r^{-}}c_{r}\overleftarrow{dB_{r}}
+\int_{t}^{T}\bar{Y}_{r^{-}}\Gamma_{s,r^{-}}|c_{r}|^2dr\\
&&-\sum_{i=1}^{m}\int_{t}^{T}\bar{Y}_{r^{-}}\Gamma_{s,r^{-}}\beta_{r}^{i}dH_{r}^{(i)}
-\int_{t}^{T}\bar{Y}_{r^{-}}\Gamma_{s,r^{-}}|c_{r}|^2dr
-\int_{t}^{T}\sum_{i=1}^{m}
\Gamma_{s,r^{-}}\beta_r^i
\bar{Z}_{r}^{(i)}dr\\
&=&\Gamma_{s,T}\bar{\xi}+\int_{t}^{T}\Gamma_{s,r^{-}}\bar{f}_{r}dr+\int_{t}^{T}\Gamma_{s,r^{-}}
\bar{h}_{r}dA_r-\sum_{i=1}^{m}\int_{t}^{T}\Gamma_{s,r^{-}}(\bar{Z}_{r}^{(i)}+\bar{Y}_{r^{-}}\beta_{r}^{i})dH_{r}^{(i)}.
\end{eqnarray*}
Taking conditional expectation w.r.t. $\mathcal{F}_s$, is not hard to see that for $s=t$
\begin{eqnarray*}
\bar{Y}_{t}&=&\E\left(\Gamma_{t,T}\bar{\xi}+
\int_t^T\Gamma_{t,r}\bar{f}_{r}dr +\int_{t}^{T}\Gamma_{t,r^{-}}\bar{h}_rdA_{r}\ | \ \mathcal{F}_{t}\right).
\end{eqnarray*}
Since, according to \eqref{adprop}, the process $\Gamma_{t,r}$ is strictly positive, we obtain $\bar{Y}_{t}\geq0$, a.s. i.e. $Y_t^1\geq Y_t^2$, a.s. Moreover if $\bar{\xi}> 0$, a.s. or $\bar{f}_t> 0$, a.s. or
$\bar{h}_t> 0$, a.s. then $\bar{Y}_{t}>0$, a.s. i.e. $Y_t^1> Y_t^2$, a.s.
\end{proof}

\section{GBDSDEL with continuous coefficients.}
In this section, we study the GBDSDEL under the continuous and linear growth condition on the coefficients. Roughly speaking, We prove the existence of a minimal or maximal solution by the well know approximation method of the functions $f$ and $h$ (Lemma $\ref{Lemma3.1}$) and the comparison theorem (Theorem $\ref{tc}$).

In addition, we give the following assumptions:
\begin{description}
\item \textbf{(H1)}\ \ The terminal value $\xi \in \mathrm{L}^{2}(\Omega,
\mathcal{F}_{T}, \mathbb{P}, \mathbb{R})$ such that for all $\lambda>0$,\ \
$\E (e^{\lambda A_T}|\xi|^2)<\infty,$
\noindent \item \textbf{(H2)}\ \ The coefficients $f:\Omega \times [0,T]\times
\mathbb{R} \times \R^{m}\rightarrow
\mathbb{R}$ and $g, h:\Omega \times [0,T]\times \mathbb{R}
\rightarrow \mathbb{R}$, satisfy, for some constants $\beta_1\in \R,\, \beta_2 < 0,\, K>0$ and  three $\mathcal{F}_t$-measurable processes
$\{f_t,\ g_t\ h_t:0\leq t\leq T\}$ with value in $[1,\infty[$
and for all $(t,y,z)\in\Omega \times [0,T]\times
\mathbb{R} \times \R^{m}$,
\begin{itemize}
\item [(i)] $f(.,y,z), g(.,y)$ and $h(.,y)$ are jointly measurable,
\item [(ii)] $|f(t,y,z)|\leq f_t+ K(|y|+\|z\|),\;\; |f(t,y,z)-f(t,y,z')|\leq K\|z-z'\|,\; \forall\, y\in\R$,
\item [(iii)] $|h(t,y)|\leq h_t+ K|y|$, for some $K>0$,
\item [(iv)] $\displaystyle{\E (\int_{0}^{T}e^{\mu t+\lambda A_t}f_t^2dt+\int_{0}^{T}e^{\mu t+\lambda A_t}g_t^2dt+\int_{0}^{T}e^{\mu t+\lambda A_t}h_t^2dA_t)<\infty}$, for all $\mu,\,\lambda>0$,
\item [(v)] $\mid g(t,y)-g(t,y') \mid^{2} \leq K\mid y-y'\mid^{2}$,
\item [(vi)] $y\mapsto f(t,y,z)$ and $y\mapsto h(t,y)$ are continuous for all $z,\omega, t$.
\end{itemize}
\end{description}
The main result of this paper is the following theorem.
\begin{theorem}\label{te}
Under assumptions $({\bf H1})$ and $({\bf H2})$, the GBDSDEL \eqref{a011}
has solution $(Y,Z)\in \mathcal{E}_{m}(0,T)$ which is a minimal one, in the sense that, if $(Y^{*},Z^{*})$ is any other solution we have
$Y^{*}\leq Y$, a.s.
\end{theorem}
To prove this theorem, we need an important result which gives an approximation
of continuous functions by Lipschitz functions (see Lepeltier and San Martin \cite{LSm} to
appear for the proof).
\begin{lemma}\label{Lemma3.1}
Let $\phi : [0,T]\times\R^p\rightarrow\R$ be a continuous function with linear growth, that is, there exists a constant
$K>0$ such that $\forall \,x\in \R^p, |\phi(t,x)|\leq \phi_t+K\|x\|$. Then the sequence of functions
\begin{eqnarray*}
\phi_n(t, x) = \inf_{y\in \Q^p}\{\phi(t,y)+n|x-y|\}
\end{eqnarray*}
is well defined for $n \geq K$ and satisfies
\begin{description}
\item $(a)$\;  Linear growth: $\forall \, (t,x)\in\times\R^p,\;\; |\phi_n(t,x)|\leq \phi_t+K\|x\|)$,
\item $(b)$ \;  Monotonicity: $\forall \, (t,x)\in\times\R^p,\;\; \phi_n(t,x)\nearrow$,
\item $(c)$\;  Lipschitz condition: $\forall \,t\in[0,T], x, y\in\R^p,\;\; |\phi_n(t,x) - \phi_n(t,y)|\leq n\|x - y\|$,
\item $(d)$\,  Strong convergence: if $x_n\rightarrow x$\; as \;$n\rightarrow\infty$, then $\phi_n(t,x_n)\rightarrow \phi(t,x)$ as $n\rightarrow\infty$ for all $t$.
\end{description}
\end{lemma}
\begin{proof}[\it Proof of theorem 3.1]
For fixed $(t,\omega)$, it follows from $({\bf H2})$ that $f(t,\omega)$ and $h(t,\omega)$ are continuous and with linear growth. Hence, by Lemma $\ref{Lemma3.1}$ there exist sequences of functions $f_n(t,\omega)$ and $h_n(t,\omega)$ associated to $f$ and $h$, respectively. Then $f_n,\,h_n$ are
measurable functions as well as Lipschitz functions. Moreover,  since $\xi$ satisfies $({\bf H1})$ we get from Ren et al. \cite{Ren} that there is a unique
pair $\{(Y^n_t ,Z^n_t), 0 \leq t\leq T\}$ of $\cal{F}_t$-measurable processes taking
values in $\R\times \R^m$ and satisfying
\begin{eqnarray}\label{30}
Y_t^n&=&\xi+\int_{t}^{T}f_n(s,Y_{s^-}^n,Z_s^n)ds+\int_{t}^{T}h_n(s,Y_{s^-}^n)dA_s+
\int_{t}^{T}g(s,Y_{s^-}^n)\overleftarrow{dB_{s}}
\notag\\&&-\sum_{i=1}^{m}\int_{t}^{T}Z_{s}^{n(i)}dH_{s}^{(i)}, \ \ \ \ t\in[0,T],
\end{eqnarray}
and
\begin{eqnarray*}
\E\left(\sup_{0\leq t\leq T}|Y^n_t|^2+\int_{0}^{T}\|Z^n_s\|^2ds\right)<\infty.
\end{eqnarray*}
Since for fixed $(t,\omega),\, f_{n+1}(t,\omega)\geq f_{n}(t,\omega),\, h_{n+1}(t,\omega)\geq h_{n}(t,\omega)$ and inequality \eqref{adprop} still holds, for all $n\geq K$, it follows from the comparison theorem (Theorem \ref{tc}) that for every $n\geq K$
\begin{eqnarray}
Y^{n}\leq Y^{n+1}, \; \;  dt\otimes d\P\mbox{-a.s.}\label{comp}
\end{eqnarray}
The idea of the proof of Theorem 3.1  is to establish that the limit of the sequence $(Y^n,Z^n)$ is a solution of the BDSDE \eqref{a011}. It follows
by the same steps and technics as in \cite{Aman} (see Theorem 3.1).\newline
{\it Step 1: A priori estimates}\newline
There exists a constant $C>0$ independent of $n$ such that
\begin{eqnarray}
\sup_{n\geq K}\E\left(\sup_{0\leq t\leq T}|Y^n_t|^2+\int_{0}^T\|Z_t^n\|^2dt\right)\leq C.\label{esti}
\end{eqnarray}
Indeed, for any $\mu,\lambda>0$, Itô's formula applied to $e^{\mu t+\lambda A_t}|Y_t^n|^2$ provides
\begin{eqnarray}
&&e^{\mu t+\lambda A_t}|Y_{t}^n|^{2}+\lambda\int_{t}^{T}e^{\mu s+\lambda A_s}|Y_{s}^{n}|^{2}dA_s
+\mu\int_{t}^{T}e^{\mu s+\lambda A_s}|Y_{s}^{n}|^{2}ds\nonumber\\
&&=e^{\mu T+\lambda A_T}|\xi|^{2}
+2\int_{t}^{T}e^{\mu s+\lambda A_s}Y_{s}^{n}f_n(s,Y_{s}^{n},Z_{s}^{n})ds+
2\int_{t}^{T}e^{\mu s+\lambda A_s}Y_{s}^{n}g(s,Y_{s}^{n})dB_{s}\nonumber \\
&&+2\int_{t}^{T}e^{\mu s+\lambda A_s}Y_{s}^{n}h_n(s,Y_{s}^{n})dA_s-
2\sum_{i=1}^{m}\int_{t}^{T}e^{\mu s+\lambda A_s}Y_{s}^{n}Z_{s}^{n(i)}dH_{s}^{(i)}+
\int_{t}^{T}e^{\mu s+\lambda A_s}|g(s,Y_{s}^{n})|^{2}ds\nonumber\\&&
-\sum_{i,j=1}^{m}\int_{t}^{T}e^{\mu s+\lambda A_s}Z_{s}^{n(i)}Z_{s}^{n(j)}d[H^{(i)},H^{(j)}]_s.\label{est}
\end{eqnarray}
Assumption $({\bf H2})$ together with Young's inequality imply, for any $\sigma>0$ et $\gamma>0$,
\begin{eqnarray*}
2Y_{s}^{n}f_n(s,Y_{s}^{n},Z_{s}^{n})&\leq&
\left(1+2K+\frac{1}{\sigma}K^2\right)\left|
Y_{s}^{n}\right|^2+\sigma\|Z_{s}^{n}\|^2+
f_s^2,\\
2Y_{s}^{n}h_n(s,Y_{s}^{n})&\leq& (2K+1)\left|
Y_{s}^{n}\right|^2+h_s^2,\hspace{4cm}\\
\left|g(s,Y_{s}^{n})\right|^2&\leq&\left(1+\gamma\right)C\left|
Y_{s}^{n}\right|^2+\left(1+\frac{1}{\gamma}\right)g_s^2.
\end{eqnarray*}
Therefore taking expectation in both side of $\eqref{est}$ with the suitable $\lambda$ and $\sigma$
\begin{eqnarray*}
&&\E \left(e^{\mu t+\lambda A_t}|Y_{t}^n|^{2}
+\int_{t}^{T}e^{\mu s+\lambda A_s}\|Z_{s}^{n}\|^{2}ds\right)\\
&&\leq C\E\left(e^{\mu T+\lambda A_T}|\xi|^{2}
+\int_{t}^{T}e^{\mu s+\lambda A_s}h_s^2dA_s+\int_{t}^{T}e^{\mu s+\lambda A_s}(
f_s^2+g_s^2)ds\right)<\infty,
\end{eqnarray*}
which by Burkhölder-Davis-Gundy's inequality provides
\begin{eqnarray*}
&&\E \left(\sup_{0\leq t\leq T}e^{\mu t+\lambda A_t}|Y_{t}^n|^{2}
+\int_{t}^{T}e^{\mu s+\lambda A_s}\|Z_{s}^{n}\|^{2}ds\right)\\
&&\leq C\E\left(e^{\mu T+\lambda A_T}|\xi|^{2}
+\int_{0}^{T}e^{\mu s+\lambda A_s}h_s^2dA_s+\int_{t}^{T}e^{\mu s+\lambda A_s}(
f_s^2+g_s^2)ds\right)<\infty.
\end{eqnarray*}
{\it Step 2: Convergence result}\newline
We have from \eqref{comp} and \eqref{esti} the existence of process $Y$ such that $Y^n_t\nearrow Y_t$ a.s. for all $t\in[0,T]$. Hence, it follows from Fatou's lemma together with the dominated convergence theorem that
\begin{eqnarray}
\E\left(\sup_{0\leq t\leq T}|Y_{t}|^{2}\right)\leq C\;\; \mbox{and}\;\; \E\left(\int_0^T|Y^n_s-Y_s|^2(ds+dA_s)\right)\rightarrow\, 0\label{conver}
\end{eqnarray}
as $n$ goes to infinity. Next, for all $n\geq n_0 \geq K$, it follows from Itô's formula, taking $t = 0$,
\begin{eqnarray*}
&&\E |Y_{0}^n-Y_{0}^{n+1}|^{2}+\E\int_{0}^{T}\|Z_{s}^n-Z_{s}^{n+1}\|^{2}ds\\
&=&2\E\int_{0}^{T}\left(Y_{s}^n-Y_{s}^{n+1}\right)\left(f_n(s,Y_{s}^n,Z_{s}^n)-f_{n+1}(s,Y_{s}^{n+1},Z_{s}^{n+1})\right)ds
\\&&+2\E\int_{0}^{T}\left(Y_{s}^n-Y_{s}^{n+1}\right)\left(h_n(s,Y_{s}^n)-h_{n+1}(s,Y_{s}^{n+1})\right)dA_s
+\E\int_{0}^{T}|g(s,Y_{s}^n)-g(s,Y_{s}^{n+1})|^{2}ds\\
&\leq& 2\Big(\E\int_{0}^{T}|Y_{s}^n-Y_{s}^{n+1}|^{2}ds\Big)^{\frac{1}{2}}
\Big(\E\int_{0}^{T}|f_n(s,Y_{s}^n,Z_{s}^n)-f_{n+1}(s,Y_{s}^{n+1},Z_{s}^{n+1})|^2ds\Big)^{\frac{1}{2}}\\
&&+2\Big(\E\int_{0}^{T}|Y_{s}^n-Y_{s}^{n+1}|^{2}dA_s\Big)^{\frac{1}{2}}
\Big(\E\int_{0}^{T}|h_n(s,Y_{s}^n)-h_{n+1}(s,Y_{s}^{n+1})|^2dA_s\Big)^{\frac{1}{2}}+
C\E\int_{0}^{T}|Y_{s}^n-Y_{s}^{n+1}|^{2}ds.
\end{eqnarray*}
The uniform linear growth condition on the sequence $(f_n, h_n)$ together with inequality \eqref{esti} provide the existence of a constant $C$
such that
\begin{eqnarray*}
\E\int_{0}^{T}\|Z_{s}^n-Z_{s}^{n+1}\|^{2}ds&\leq& C'\Big(\E\int_{0}^{T}|Y_{s}^n-Y_{s}^{n+1}|^{2}(ds+dA_s)\Big)^{\frac{1}{2}}.
\end{eqnarray*}
Thus from \eqref{conver}, $\{Z^n\}$ is a Cauchy sequence in a Banach space $\mathcal{M}^{2}(0,T,\R^{m})$, and there exists an $\cal{F}_t$-jointly
measurable process $Z$ such that $\{Z^n\}$ converges to $Z$ as $n\rightarrow \infty$.

Similarly, by Itô's formula together with Burkholder-Davis-Gundy inequality, it follows
that
\begin{eqnarray*}
\E\left(\sup_{0\leq t\leq T}|Y_{t}^n-Y_{t}^{n+1}|^{2}\right)\rightarrow \, 0\;\; \mbox{as}\;\; n\rightarrow\infty,
\end{eqnarray*}
from which we deduce that $\P$-almost surely, $Y^n$ converges uniformly to $Y$ which is continuous.

{\it Step 3: $(Y,Z)$ verifies GBDSDEL $\eqref{a011}$} \newline
Since $Z^n\rightarrow Z$ in $\cal{M}^2(0,T,\R^m)$, along a subsequence which we still denote $Z^n,\;\; Z^n\rightarrow Z, \;  dt\otimes d\P\;\; \mbox{a.e}$
and there exists $\Pi\in\mathcal{M}^{2}(0,T,\R^{m})$ such that $\forall n, |Z^n|<\Pi,\; dt\otimes d\P\;\; \mbox{a.e}$. Therefore, by Lemma \ref{Lemma3.1}, we have
\begin{eqnarray*}
f_{n}(t,Y_{t}^{n},Z_{t}^{n})&\rightarrow& f(t,Y_{t},Z_{t})\; dt\otimes d\P\;\; \mbox{a.e.},\\
h_{n}(t,Y_{t}^{n})&\rightarrow& h(t,Y_{t})\;\; dA_t\otimes d\P\;\; \mbox{a.e.}
\end{eqnarray*}
Moreover, from $({\bf H2})$ and \eqref{esti}, the dominated convergence theorem provides
\begin{eqnarray*}
\E\left(\int_{t}^{T}f_{n}(t,Y_{t}^{n},Z_{t}^{n})ds\right)
&\rightarrow &\E\left(\int_{t}^{T}f(t,Y_{t},Z_{t})ds\right),\\
\E\left(\int_{t}^{T}h_{n}(t,Y_{t}^{n})dA_s\right)&\rightarrow&\E\left(\int_{t}^{T}h(t,Y_{t})dA_s\right)
\end{eqnarray*}
as $n\rightarrow\infty$.
Further, in virtue of Burkholder-Davis-Gundy inequality, $({\bf H2})$ and \eqref{conver}, we obtain
\begin{eqnarray*}
\E\left(\underset{0\leq t\leq T}\sup\left|\int_{t}^{T}g(s,Y_{s}^{n})dB_s- \int_{t}^{T}g(s,Y_{s})dB_s \right|\right)
&\rightarrow 0 &\\
\E\left(\underset{0\leq t\leq T}\sup\left|\sum_{i=1}^{m}\left(\int_{t}^{T}Z_{s}^{n(i)}dH_s^{(i)}- \int_{t}^{T}Z_{s}^{(i)}dH_s^{(i)} \right)\right|\right)&\rightarrow
& 0
\end{eqnarray*}
as $n$ goes to infinity.
Finally, passing to the limit in \eqref{30}, we conclude that $(Y,Z)$ is a solution of GBDSDEL \eqref{a011}.

{\it Step 4: Minimal solution}\newline
Let $(Y',Z') \in \mathcal{E}^{2}_{m}(0,T)$ be any solution of GBDSDEL \ref{a011}.
By virtue of the comparison theorem (Theorem \ref{tc}), we have $Y^n\leq Y',\ \forall n\in\N$. Therefore, $Y\leq Y'$.
That proves that $Y$ is the minimal solution.
\end{proof}
\begin{remark}
Using the same arguments and the following approximating sequence
\begin{eqnarray*}
\phi_n(t, x) =\sup_{y\in\Q^p}\{\phi(t,y) - n|x - y|\},
\end{eqnarray*}
one can prove that the GBDSDEL $\eqref{a011}$ has a maximal solution
\end{remark}

\noindent{\bf Acknowledgments}\newline The authors thanks an anonymous
referee for his comments, remarks and for a
significant improvement to the overall presentation of this paper.

\end{document}